\documentclass{amsart}%
\usepackage{amsmath}
\usepackage{amsfonts}
\usepackage{amssymb}
\usepackage{amscd}
\usepackage{graphicx}%
\providecommand{\U}[1]{\protect \rule{.1in}{.1in}}
\newtheorem{theorem}{Theorem}[section]
\newtheorem{lemma}[theorem]{Lemma}

\newtheorem{corollary}[theorem]{Corollary}

\begin{document}
\title{Compactness of matrix operators on some sequence spaces derived by Fibonacci numbers}
\author{Emrah Evren KARA}
\address{DEPARTMENT OF MATHEMATICS, D\"{U}ZCE UNIVERSITY, 81620, D\"{U}ZCE, TURKEY}
\email{eevrenkara@hotmail.com, eevrenkara@duzce.edu.tr}
\author{Met\.{I}n Ba\c{s}ar\i r}
\address{DEPARTMENT OF MATHEMATICS, SAKARYA UNIVERSITY, 54187, SAKARYA, TURKEY}
\email{basarir@sakarya.edu.tr}
\author{M. Mursaleen}
\address{DEPARTMENT OF MATHEMATICS, ALIGARH MUSLIM UNIVERSITY, 202002, ALIGARH, INDIA}
\email{mursaleenm@gmail.com}
\date{}

\begin{abstract}
In this paper, we apply the Hausdorff measure of noncompactness to obtain the
necessary and sufficient conditions for certain matrix operators on the
Fibonacci difference sequence spaces $\ell_{p}(\widehat{F})$ and $\ell
_{\infty}(\widehat{F})$ to be compact, where $1\leq p<\infty$.

\end{abstract}
\maketitle

\section{\textbf{Introduction and preliminaries}}

Let $%
\mathbb{N}
=\{0,1,2,...\}$ and $%
\mathbb{R}
$ be the set of all real numbers. We shall write $\lim_{k}$, $\sup_{k}$,
$\inf_{k}$ and $\sum_{k}$ instead of $\lim_{k\rightarrow \infty}$, $\sup_{k\in%
\mathbb{N}
}$, $\inf_{k\in%
\mathbb{N}
}$ and $\sum_{k=0}^{\infty}$, respectively. Let $\omega$ be the vector space
of all real sequences $x=(x_{k})_{k\in%
\mathbb{N}
}$. By the term $\mathit{sequence}$ $\mathit{space}$, we shall mean any linear
subspace of $\omega$. Let $\varphi,$ $\ell_{\infty},$ $c$ and $c_{0}$ denote
\ the sets \ of all finite, bounded, convergent and null sequences,
respectively. We write $\ell_{p}=\{x\in \omega:\sum_{k}\left \vert
x_{k}\right \vert ^{p}<\infty \}$ for $1\leq p<\infty.$ Also, we shall use the
conventions that $e=(1,1,...)$ and $e^{(n)}$ is the sequence whose only
non-zero term is $1$ in the $n^{\text{th}}$ place for each $n\in%
\mathbb{N}
$. For any sequence $x=(x_{k})$, let $x^{[n]}=\sum_{k=0}^{n}x_{k}e^{(k)}$ be
its $n$-section. Morever, we write $bs$ and $cs$ for the sets of sequences
with bounded and convergent partial sums, respectively.

The $\mathit{Fibonacci}$ $\mathit{numbers}$ are the sequence of numbers
$\{f_{n}\}_{n=0}^{\infty}$ defined by the linear recurrence equations
\[
f_{0}=f_{1}=1\text{ and }f_{n}=f_{n-1}+f_{n-2}\text{; \ }n\geq2.
\]

Fibonacci numbers have many interesting properties and applications in arts,
sciences and architecture. For example, the ratio sequences of Fibonacci
numbers converge to the golden ratio which is important in sciences and arts.
Also, some basic properties of Fibonacci numbers can be found in [2].

A $\mathit{B-space}$\textit{ }is a complete normed space. A topological
sequence space in which all coordinate functionals $\pi_{k}$, $\pi
_{k}(x)=x_{k}$, are continuous is called a $\mathit{K-}$ $\mathit{space}$. A
$\mathit{BK-}$ $\mathit{space}$ is defined as a $K-$ space which is also a
$B-$ space, that is, a $BK-$ space is a Banach space with continuous
coordinates. A $BK-$ space $X\supset \varphi$ is said to have $AK$ if every
sequence $x=(x_{k})\in X$ has a unuqiue representation $x=\sum_{k}x_{k}%
e^{(k)}$. For example, the space $\ell_{p}$ $(1\leq p<\infty)$ is $BK-$ space
with $\left \Vert x\right \Vert _{p}=\left(  \sum_{k}\left \vert x_{k}\right \vert
^{p}\right)  ^{1/p}$ and $c_{0}$, $c$ and $\ell_{\infty}$ are $BK-$ spaces
with $\left \Vert x\right \Vert _{\infty}=\sup_{k}\left \vert x_{k}\right \vert $.
Further, the $BK-$ spaces $c_{0}$ and $\ell_{p}$ have $AK$, where $1\leq
p<\infty$ (cf. [3,4].

A sequence $(b_{n})$ in a normed space $X$ is called a $\mathit{Schauder}%
$\textit{ }$\mathit{basis}$ for $X$ if for every $x\in X$ there is a unique
sequence $(\alpha_{n})$ of scalars such that $x=\sum_{n}\alpha_{n}b_{n}$,
i.e., $\lim_{m}\left \Vert x-\sum_{n=0}^{m}\alpha_{n}b_{n}\right \Vert =0.$

The $\beta$-dual of a sequence space $X$ is defined by
\[
X^{\beta}=\left \{  a=(a_{k})\in \omega:ax=(a_{k}x_{k})\in cs\text{ for all
}x=(x_{k})\in X\right \}  .
\]

Let $A=(a_{nk})_{n,k=0}^{\infty}$ be an infinite matrix of real numbers
$a_{nk}$, where $n,k\in%
\mathbb{N}
$. We write $A_{n}$ for the sequence in the $n^{\text{th}}$ row of $A$, that
is $A_{n}=(a_{nk})_{k=0}^{\infty}$ for every $n\in%
\mathbb{N}
$. In addition, if $x=(x_{k})_{k=0}^{\infty}\in \omega$ then we define the
$A$\textit{-}$\mathit{transform}$ $\mathit{of}$ $x$ as the sequence
$Ax=\left \{  A_{n}(x)\right \}  _{n=0}^{\infty}$, where
\begin{equation}
A_{n}(x)=%
{\displaystyle \sum_{k=0}^{\infty}}
a_{nk}x_{k};\text{ \  \ }\left(  n\in%
\mathbb{N}
\right)  \tag{1.1}%
\end{equation}
provided the series on the right side converges for each $n\in%
\mathbb{N}
.$

For arbitrary subsets $X$ and $Y$ of $\omega$, we write $\left(  X,Y\right)  $
for the class of all infinite matrices that map $X$ into $Y$. Thus,
$A\in \left(  X,Y\right)  $ if and only if $A_{n}\in X^{\beta}$ for all $n\in%
\mathbb{N}
$ and $Ax\in Y$ for all $x\in X$.

The matrix domain $X_{A}$ of an infinite matrix $A$ in sequence space $X$ is
defined by%
\begin{equation}
X_{A}=\left \{  x=\left(  x_{k}\right)  \in \omega:Ax\in X\right \}  \tag{1.2}%
\end{equation}
which is a sequence space.

Let $\Delta$ denotes the matrix $\Delta=(\Delta_{nk})$ defined by%
\[
\Delta_{nk}=\left \{
\begin{array}
[c]{cc}%
(-1)^{n-k} & (n-1\leq k\leq n)\\
0 & (0\leq k<n-1)\text{ or\ }(k>n)
\end{array}
\right.
\]
or%
\[
\Delta_{nk}=\left \{
\begin{array}
[c]{cc}%
(-1)^{n-k} & (n\leq k\leq n+1)\\
0 & (0\leq k<n)\text{ or\ }(k>n+1).
\end{array}
\right.
\]

In the literature, the matrix domain $\lambda_{\Delta}$ is called the
$\mathit{difference}$\textit{ }$\mathit{sequence}$\textit{ }$\mathit{space}$
whenever $\lambda$ is a normed or paranormed sequence space. The idea of
difference sequence space was introduced by K\i zmaz [5]. In 1981, K\i zmaz
[5] defined the sequence spaces%
\[
X(\Delta)=\left \{  x=(x_{k})\in \omega:(x_{k}-x_{k+1})\in X\right \}
\]
for $X=\ell_{\infty}$, $c$ and $c_{0}$. The difference space $bv_{p}$,
consisting of all sequnces $(x_{k})$ such that $(x_{k}-x_{k-1})$ is in the
sequence space $\ell_{p}$, was studied in the case $0<p<1$ by Altay and
Ba\c{s}ar [6] and in the case $1\leq p\leq \infty$ by Ba\c{s}ar and Altay [7]
and \c{C}olak et al. [8]. The paranormed difference sequence space%
\[
\Delta \lambda(p)=\{x=\left(  x_{k}\right)  \in \omega:(x_{k}-x_{k+1})\in
\lambda(p)\}
\]
was examined by Ahmad and Mursaleen [9] and Malkowsky [10], where $\lambda(p)$
is any of the paranormed spaces $\ell_{\infty}(p)$, $c(p)$ and $c_{0}(p)$
defined by Simons [11] and Maddox [12].

Recently, Ba\c{s}ar et al. [13] have defined the sequence spaces $bv(u,p)$ and
$bv_{\infty}(u,p)$ by%
\[
bv(u,p)=\{x=\left(  x_{k}\right)  \in \omega:\sum_{k}\left \vert u_{k}%
(x_{k}-x_{k-1})\right \vert ^{p_{k}}<\infty \}
\]
and
\[
bv_{\infty}(u,p)=\{x=\left(  x_{k}\right)  \in \omega:\sup_{k\in%
\mathbb{N}
}\left \vert u_{k}(x_{k}-x_{k-1})\right \vert ^{p_{k}}<\infty \},
\]
where $u=(u_{k})$ is an arbitrary fixed sequence and $0<p_{k}\leq H<\infty$
for all $k\in%
\mathbb{N}
.$ Also in [14-23], authors studied some difference sequence spaces.

Let $S_{X}$ denote the unit sphere in a normed linear space $X$. If
$X\supset \varphi$ is a $BK$ space and $a=(a_{k})\in \omega$, then we write
\begin{equation}
\left \Vert a\right \Vert _{X}^{\ast}=\sup_{x\in S_{X}}\left \vert \sum
\limits_{k}a_{k}x_{k}\right \vert \tag{1.3}%
\end{equation}
provided the expression on the right side is defined and finite which is the
case whenever $a\in X^{\beta}$.

The following results are very important in our study.

\begin{lemma}
[{[3, Theorem 1.29]}]Let $1<p<\infty$ and $q=p/(p-1)$. Then, we have
$\ell_{\infty}^{\beta}=\ell_{1}$, $\ell_{1}^{\beta}=\ell_{\infty}$ and
$\ell_{p}^{\beta}=\ell_{q}$. Furthermore, let $X$ denote any of the spaces
$\ell_{\infty},$ $\ell_{1}$ or $\ell_{p}$. Then, we have $\left \Vert
a\right \Vert _{X}^{\ast}=\left \Vert a\right \Vert _{X^{\beta}}$ for all $a\in
X^{\beta}$, where $\left \Vert .\right \Vert _{X^{\beta}}$ is the natural norm
on the dual space $X^{\beta}$.
\end{lemma}

\begin{lemma}
[{[3, Theorem 1.23 (a)]}]Let $X$ and $Y$ be $BK$-spaces. Then we have
$(X,Y)\subset B(X,Y)$, that is, every matrix $A\in(X,Y)$ defines a linear
operator $L_{A}\in B(X,Y)$ by $L_{A}(x)=Ax$ for all $x\in X$, where $B(X,Y)$
denotes the set all bounded (continuous) linear operators $L:X\rightarrow Y.$
\end{lemma}

\begin{lemma}
[{[3, Lemma 2.2]}]Let $X\supset \phi$ be $BK$-space and $Y$ be any of the
spaces $c_{0},$ $c$ or $\ell_{\infty}$. If $A\in(X,Y)$, then%

\[
\left \Vert L_{A}\right \Vert =\left \Vert A\right \Vert _{(X,\ell_{\infty}%
)}=\underset{n}{\sup}\left \Vert A_{n}\right \Vert _{X}^{\ast}<\infty.
\]

\end{lemma}

By $M_{X},$ we denote the collection of all bounded subsets of a metric space
$\left(  X,d\right)  .$ If $Q\in M_{X},$ then the \textit{Hausdorff measure of
noncompactness} of the set $Q,$ denoted by $\chi \left(  Q\right)  ,$ is
defined by
\[
\chi \left(  Q\right)  :=\inf \left \{  \varepsilon>0:Q\subset \underset
{i=1}{\overset{n}{\cup}}B\left(  x_{i},r_{i}\right)  ,\text{ }x_{i}\in
X,\text{ }r_{i}<\varepsilon \text{ }\left(  i=1,2,...,n\right)  ,\text{ }n\in%
\mathbb{N}
-\{0\} \right \}  .
\]
The function $\chi:M_{X}\rightarrow \left[  0,\infty \right)  $ is called the
$\mathit{Hausdorff}$\textit{ }$\mathit{measure}$\textit{ }$\mathit{of}%
$\textit{ }$\mathit{noncompactness}$.

The basic properties of the Hausdorff measure of noncompactness can be found
in [3]

The following result gives an estimate for the Hausdorff measure of
noncompactness in the $BK$ space $\ell_{p}$ for $1\leq p<\infty.$

\begin{lemma}
[{[24, Theorem 2.8]}]Let $1\leq p<\infty$ and $Q\in M_{\ell_{p}}.$ If
$P_{m}:\ell_{p}\rightarrow \ell_{p}$ $(m\in%
\mathbb{N}
)$ is the operator defined by $P_{m}(x)=(x_{0},x_{1},...,x_{m},0,0,...)$ for
all $x=(x_{k})\in \ell_{p}$, then we have%
\[
\chi(Q)=\lim_{m\rightarrow \infty}\left(  \sup_{x\in Q}\left \Vert
(I-P_{m})(x)\right \Vert _{\ell_{p}}\right)  ,
\]
where $I$ is the identity operator on $\ell_{p}.$
\end{lemma}

Let $X$ and $Y$ be Banach spaces. Then, a linear operator $L:X\rightarrow Y$
is said to be $\mathit{compact}$ if the domain of $L$ is all of $X$ and $L(Q)$
is a totally bounded subset of $Y$ for every $Q\in M_{X}$. Equivalently, we
say that $L$ is compact if its domain is all of $X$ and for every bounded
sequence $\left(  x_{n}\right)  $ in $X,$ the sequence $\left(  L\left(
x_{n}\right)  \right)  $ has a convergent subsequence in $Y.$

The idea of compact operators between Banach spaces is closely related to the
Hausdorff measure of noncompactness, and it can be given as follows:

Let $X$ and $Y$ be Banach spaces and $L\in B(X,Y)$. Then, the Hausdorff
measure of noncompactness of $L$, is denoted by $\left \Vert L\right \Vert
_{\chi}$, can be given by%
\begin{equation}
\left \Vert L\right \Vert _{\chi}=\chi(L(S_{X})) \tag{1.4}%
\end{equation}
and we have
\begin{equation}
L\text{ is compact if and only if }\left \Vert L\right \Vert _{\chi}=0.
\tag{1.5}%
\end{equation}

The Hausdorff measure of noncompactness has various applications in the theory
of sequence spaces, one of them is to obtain necessary and sufficient
conditions for matrix operators between $BK$ spaces to be compact. Recently,
several authors have studied compact operators on the sequence spaces and
given very important results related to the Hausdorff measure of
noncompactness of a linear operator. For example [25-38].

In this paper, we derive some identities for the Hausdorff measure of
noncompactness on the Fibonacci difference sequence spaces $\ell_{p}%
(\widehat{F})$ and $\ell_{\infty}(\widehat{F})$ defined by Kara [1]. We also
apply the Hausdorff measure of noncompactness to obtain the necessary and
sufficient conditions for such operators to be compact.

\section{\textbf{\ The Fibonacci Difference Sequence Spaces }$\ell
_{p}(\widehat{F})$\textbf{ and }$\ell_{\infty}(\widehat{F})$}

Throughout, let $1\leq p\leq \infty$ and $q$ denote the conjugate of $p$, that
is, $q=p/(p-1)$ for $1<p<\infty$, $q=\infty$ for $p=1$ or $q=1$ for
$p=\infty.$

Recently, Kara[1] has defined the Fibonacci difference sequence spaces
$\ell_{p}(\widehat{F})$ and $\ell_{\infty}(\widehat{F})$ by%
\[
\ell_{p}(\widehat{F})=\left \{  x=(x_{n})\in \omega:%
{\displaystyle \sum \limits_{n}}
\left \vert \frac{f_{n}}{f_{n+1}}x_{n}-\frac{f_{n+1}}{f_{n}}x_{n-1}\right \vert
^{p}<\infty \right \}  ;\text{ }1\leq p<\infty
\]
and%
\[
\ell_{\infty}(\widehat{F})=\left \{  x=(x_{n})\in \omega:\sup_{n\in%
\mathbb{N}
}\left \vert \frac{f_{n}}{f_{n+1}}x_{n}-\frac{f_{n+1}}{f_{n}}x_{n-1}\right \vert
<\infty \right \}  .\text{ \  \  \  \  \  \  \  \  \  \  \  \  \  \  \  \  \ }%
\]

With the notation of (1.2), the sequence spaces $\ell_{p}(\widehat{F})$ and
$\ell_{\infty}(\widehat{F})$ may be redefined by
\begin{equation}
\ell_{p}(\widehat{F})=(\ell_{p})_{\widehat{F}}\text{ }(1\leq p<\infty)\text{
\ and \ }\ell_{\infty}(\widehat{F})=(\ell_{\infty})_{\widehat{F}}, \tag{2.1}%
\end{equation}
where the matrix $\widehat{F}=(\widehat{f}_{nk})$ is defined by%
\begin{equation}
\widehat{f}_{nk}=\left \{
\begin{array}
[c]{cc}%
-\frac{f_{n+1}}{f_{n}} & (k=n-1)\\
\frac{f_{n}}{f_{n+1}} & (k=n)\\
0 & (0\leq k<n-1)\text{ or }(k>n)
\end{array}
\right.  ;\text{ }(n,k\in%
\mathbb{N}
). \tag{2.2}%
\end{equation}
Further, it is clear that the spaces $\ell_{p}(\widehat{F})$ and $\ell
_{\infty}(\widehat{F})$ are\ $BK$ spaces with the norms given by
\begin{equation}
\left \Vert x\right \Vert _{\ell_{p}(\widehat{F})}=\left(  \sum_{n}\left \vert
y_{n}(x)\right \vert ^{p}\right)  ^{1/p}\text{ ; \ }(1\leq p<\infty)\text{
\ and }\left \Vert x\right \Vert _{\ell_{\infty}(\widehat{F})}=\sup_{n\in%
\mathbb{N}
}\left \vert y_{n}(x)\right \vert , \tag{2.3}%
\end{equation}
where the sequence $y=(y_{n})=(\widehat{F}_{n}(x))$ is the $\widehat{F}%
$-transform of a sequence $x=(x_{n})$, i.e.,
\begin{equation}
y_{n}=\widehat{F}_{n}(x)=\left \{
\begin{array}
[c]{cc}%
\frac{f_{0}}{f_{1}}x_{0}=x_{0}\text{ \  \  \  \ } & (n=0)\\
\frac{f_{n}}{f_{n+1}}x_{n}-\frac{f_{n+1}}{f_{n}}x_{n-1} & (n\geq1)
\end{array}
\right.  \text{ };\text{ }(n\in%
\mathbb{N}
). \tag{2.4}%
\end{equation}

Moreover, it is obvious by (2.2) that $\widehat{F}$ is a triangle. Thus, it
has a unique inverse $\widehat{F}^{-1}$ which is also a triangle and the
entries of $\widehat{F}^{-1}$ are given by%
\begin{equation}
\widehat{f}_{nk}^{-1}=\left \{
\begin{array}
[c]{cc}%
\frac{f_{n+1}^{2}}{f_{k}f_{k+1}} & (0\leq k\leq n)\\
0 & (k>n)
\end{array}
\right.  \tag{2.5}%
\end{equation}
for all $n,k\in%
\mathbb{N}
$. Therefore, we have by (2.4) that%
\begin{equation}
x_{n}=\sum_{k=0}^{n}\frac{f_{n+1}^{2}}{f_{k}f_{k+1}}y_{k}\text{ ; }(n\in%
\mathbb{N}
). \tag{2.6}%
\end{equation}

In [1], the $\beta$-duals of the sequence spaces $\ell_{p}(\widehat{F})$
$(1\leq p<\infty)$ and $\ell_{\infty}(\widehat{F})$ have been determined and
some related matrix classes characterized. Now, by taking into account that
the inverse of $\widehat{F}$ is given by (2.5), we have the following lemma
which is immediate by [1, Theorem 4.6].

\begin{lemma}
Let $1\leq p\leq \infty$. If $a=(a_{k})\in \{ \ell_{p}(\widehat{F})\}^{\beta}$,
then $\bar{a}=(\bar{a}_{k})\in \ell_{q}$ and we have
\begin{equation}
\sum \limits_{k}a_{k}x_{k}=\sum \limits_{k}\bar{a}_{k}y_{k} \tag{2.7}%
\end{equation}
for all $x=(x_{k})\in \ell_{p}(\widehat{F})$ with $y=\widehat{F}x$, where
\begin{equation}
\bar{a}_{k}=%
{\displaystyle \sum \limits_{j=k}^{\infty}}
\frac{f_{j+1}^{2}}{f_{k}f_{k+1}}a_{j}\text{ ; }(k\in%
\mathbb{N}
). \tag{2.8}%
\end{equation}

\end{lemma}

Now, we prove the following results which will be needed in the sequel.

\begin{lemma}
Let $1<p<\infty,$ $q=p/(p-1)$ and $\bar{a}=(\bar{a}_{k})$ be the sequence
defined by (2.8) Then, we have

(a) If $a=(a_{k})\in \{ \ell_{\infty}(\widehat{F})\}^{\beta}$, then $\left \Vert
a\right \Vert _{\ell_{\infty}(\widehat{F})}^{\ast}=\sum_{k}\left \vert \bar
{a}_{k}\right \vert <\infty.$

(b) If $a=(a_{k})\in \{ \ell_{1}(\widehat{F})\}^{\beta}$, then $\left \Vert
a\right \Vert _{\ell_{1}(\widehat{F})}^{\ast}=\sup_{k}\left \vert \bar{a}%
_{k}\right \vert <\infty.$

(c) If $a=(a_{k})\in \{ \ell_{p}(\widehat{F})\}^{\beta}$, then $\left \Vert
a\right \Vert _{\ell_{p}(\widehat{F})}^{\ast}=\left(  \sum_{k}\left \vert
\bar{a}_{k}\right \vert ^{q}\right)  ^{1/q}<\infty.$
\end{lemma}

\begin{proof}
(a) Let $a=(a_{k})\in \{ \ell_{\infty}(\widehat{F})\}^{\beta}$. Then, it
follows by Lemma 2.1 that $\bar{a}=(\bar{a}_{k})\in \ell_{1}$ and the equality
(2.7) holds for all sequences $x=(x_{k})\in \ell_{\infty}(\widehat{F})$ and
$y=(y_{k})\in \ell_{\infty}$ which are connected by the relation $y=\widehat
{F}x$. Further, we have by (2.3) that $x\in S_{\ell_{\infty}(\widehat{F})}$ if
and only if $y\in S_{\ell_{\infty}}$. Therefore, we derive from (1.3) and
(2.7) that%
\[
\left \Vert a\right \Vert _{\ell_{\infty}(\widehat{F})}^{\ast}=\sup_{x\in
S_{\ell_{\infty}(\widehat{F})}}\left \vert \sum \limits_{k}a_{k}x_{k}\right \vert
=\sup_{y\in S_{\ell_{\infty}}}\left \vert \sum \limits_{k}\bar{a}_{k}%
y_{k}\right \vert =\left \Vert \bar{a}\right \Vert _{\ell_{\infty}}^{\ast}.
\]
Hence, by using Lemma 1.1, we have that%
\[
\left \Vert a\right \Vert _{\ell_{\infty}(\widehat{F})}^{\ast}=\left \Vert
\bar{a}\right \Vert _{\ell_{\infty}}^{\ast}=\left \Vert \bar{a}\right \Vert
_{\ell_{1}}.
\]
This completes the proof of part (a).

Since parts (b) and (c) can also be proved by analogy with part (a), we leave
the detailed proof to the reader.
\end{proof}

Throughout this paper, if $A=(a_{nk})$ is an infinite matrix, we define the
associated matrix defined $\bar{A}=(\bar{a}_{nk})$\ by
\begin{equation}
\bar{a}_{nk}=%
{\displaystyle \sum \limits_{j=k}^{\infty}}
\frac{f_{j+1}^{2}}{f_{k}f_{k+1}}a_{nj}\text{; }(n,k\in%
\mathbb{N}
) \tag{2.9}%
\end{equation}
provided the series on the right side converges for all $n,k\in%
\mathbb{N}
$ which is the case whenever $A_{n}\in \{ \ell_{p}(\widehat{F})\}^{\beta}$ for
all $n\in%
\mathbb{N}
$, where $1\leq p\leq \infty$. Then, we have:

\begin{lemma}
Let $1\leq p\leq \infty$, $X$ be a sequence space and $A=(a_{nk})$ be an
infinite matrix. If $A\in(\ell_{p}(\widehat{F}),X)$, then $\bar{A}\in(\ell
_{p},X)$ such that $Ax=\bar{A}y$ for all $x\in \ell_{p}(\widehat{F})$ with
$y=\widehat{F}x,$ where $\bar{A}=(\bar{a}_{nk})$ is the associated matrix
defined by (2.9).
\end{lemma}

\begin{proof}
Suppose that $A\in(\ell_{p}(\widehat{F}),X)$ and let $x\in \ell_{p}(\widehat
{F})$. Then $A_{n}\in \{ \ell_{p}(\widehat{F})\}^{\beta}$ for all $n\in%
\mathbb{N}
$. Thus, it follows by Lemma 2.1 that $\bar{A}_{n}\in \ell_{q}$ for all $n\in%
\mathbb{N}
$ and the equality $Ax=\bar{A}y$ holds which yields that $\bar{A}y\in X$,
where $y=\widehat{F}x.$ Since every $y\in \ell_{p}$ is the assocaited sequence
of some $x\in \ell_{p}(\widehat{F})$, we obtain that $\bar{A}\in(\ell_{p},X)$.
This concludes the proof.
\end{proof}

\begin{lemma}
Let $1\leq p<\infty.$ If $A\in(\ell_{1}(\widehat{F}),\ell_{p})$, then
\[
\left \Vert L_{A}\right \Vert =\left \Vert A\right \Vert _{(\ell_{1}(\widehat
{F}),\ell_{p})}=\sup_{k}\left(
{\displaystyle \sum \limits_{n}}
\left \vert \bar{a}_{nk}\right \vert ^{p}\right)  ^{1/p}<\infty.
\]

\end{lemma}

\begin{proof}
The proof is elementary and left to the reader.
\end{proof}

\section{\textbf{Compact Operators on the spaces }$\ell_{p}(\widehat{F}%
)$\textbf{\ and }$\ell_{\infty}(\widehat{F})$}

In this section, we give some classes of compact operators on the spaces
$\ell_{p}(\widehat{F})$ and $\ell_{\infty}(\widehat{F}),$ where $1\leq
p<\infty$.

The following lemma gives necessary and sufficient conditions for a matrix
transformation from a $BK$ space $X$ to $c_{0}$, $c$ and $\ell_{\infty}$ to be
compact (the only sufficient condition for $\ell_{\infty}$).

\begin{lemma}
[{[25, Theorem 3.7]}]Let $X\supset \varphi$ be a BK space. Then, we have

(a) If $A\in(X,\ell_{\infty})$, then%
\[
0\leq \left \Vert L_{A}\right \Vert _{\chi}\leq \underset{n}{\lim \sup}\left \Vert
A_{n}\right \Vert _{X}^{\ast}.
\]

(b) If $A\in(X,c_{0})$, then%
\[
\left \Vert L_{A}\right \Vert _{\chi}=\underset{n}{\lim \sup}\left \Vert
A_{n}\right \Vert _{X}^{\ast}.
\]

(c) If $X$ has $AK$ or $X=\ell_{\infty}$ and $A\in(X,c)$, then
\[
\frac{1}{2}.\underset{n}{\lim \sup}\left \Vert A_{n}-\alpha \right \Vert
_{X}^{\ast}\leq \left \Vert L_{A}\right \Vert _{\chi}\leq \underset{n}{\lim \sup
}\left \Vert A_{n}-\alpha \right \Vert _{X}^{\ast},
\]
where $\alpha=(\alpha_{k})$ with $\alpha_{k}=\lim_{n}a_{nk}$ for all $k\in%
\mathbb{N}
$.
\end{lemma}

Now, let $A=(a_{nk})$ be an infinite matrix and $\bar{A}=(\bar{a}_{nk})$ the
associated matrix defined by (2.9). Then, by combining Lemmas 2.2, 2.3 and
3.1, we have the following result:

\begin{theorem}
Let $1<p<\infty$ and $q=p/(p-1)$. Then we have:

(a) If $A\in(\ell_{p}(\widehat{F}),\ell_{\infty})$, then%
\begin{equation}
0\leq \left \Vert L_{A}\right \Vert _{\chi}\leq \underset{n}{\lim \sup}\left(
{\displaystyle \sum \limits_{k}}
\left \vert \bar{a}_{nk}\right \vert ^{q}\right)  ^{1/q} \tag{3.1}%
\end{equation}
and
\begin{equation}
L_{A}\text{ is compact if }\underset{n}{\lim}\left(
{\displaystyle \sum \limits_{k}}
\left \vert \bar{a}_{nk}\right \vert ^{q}\right)  ^{1/q}=0. \tag{3.2}%
\end{equation}

(b) If $A\in(\ell_{p}(\widehat{F}),c_{0})$, then
\begin{equation}
\left \Vert L_{A}\right \Vert _{\chi}=\underset{n}{\lim \sup}\left(
{\displaystyle \sum \limits_{k}}
\left \vert \bar{a}_{nk}\right \vert ^{q}\right)  ^{1/q} \tag{3.3}%
\end{equation}
and
\begin{equation}
L_{A}\text{ is compact if and only if }\underset{n}{\lim}\left(
{\displaystyle \sum \limits_{k}}
\left \vert \bar{a}_{nk}\right \vert ^{q}\right)  ^{1/q}=0. \tag{3.4}%
\end{equation}

(c) If $A\in(\ell_{p}(\widehat{F}),c)$, then%
\begin{equation}
\frac{1}{2}.\underset{n}{\lim \sup}\left(
{\displaystyle \sum \limits_{k}}
\left \vert \bar{a}_{nk}-\bar{\alpha}_{k}\right \vert ^{q}\right)  ^{1/q}%
\leq \left \Vert L_{A}\right \Vert _{\chi}\leq \underset{n}{\lim \sup}\left(
{\displaystyle \sum \limits_{k}}
\left \vert \bar{a}_{nk}-\bar{\alpha}_{k}\right \vert ^{q}\right)  ^{1/q}
\tag{3.5}%
\end{equation}
and
\begin{equation}
L_{A}\text{ is compact if and only if }\underset{n}{\lim}\left(
{\displaystyle \sum \limits_{k}}
\left \vert \bar{a}_{nk}-\bar{\alpha}_{k}\right \vert ^{q}\right)  ^{1/q}=0,
\tag{3.6}%
\end{equation}
where $\bar{\alpha}=(\bar{\alpha}_{k})$ with $\bar{\alpha}_{k}=\lim_{n}\bar
{a}_{nk}$ for all $k\in%
\mathbb{N}
$.
\end{theorem}

\begin{proof}
It is obvious that (3.2), (3.4) and (3.6) are respectively obtained from
(3.1), (3.3) and (3.5) by using (1.5). Thus, we have to proof (3.1), (3.3) and (3.5).

Let $A\in(\ell_{p}(\widehat{F}),\ell_{\infty})$ or $A\in(\ell_{p}(\widehat
{F}),c_{0}).$ Since $A_{n}\in \{ \ell_{p}(\widehat{F})\}^{\beta}$ for all $n\in%
\mathbb{N}
$, we have from Lemma 2.2(c) that%
\begin{equation}
\left \Vert A_{n}\right \Vert _{\ell_{p}(\widehat{F})}^{\ast}=\left \Vert \bar
{A}_{n}\right \Vert _{\ell_{p}}^{\ast}=\left(
{\displaystyle \sum \limits_{k}}
\left \vert \bar{a}_{nk}\right \vert ^{q}\right)  ^{1/q} \tag{3.7}%
\end{equation}
for all $n\in%
\mathbb{N}
$. Hence, by using the equation in (3.7), we get (3.1) and (3.3) from parts
(a) and (b) of Lemma 3.1, respectively.

To prove (3.5), we have $A\in(\ell_{p}(\widehat{F}),c)$ and hence $\bar{A}%
\in(\ell_{p},c)$ by Lemma 2.3. Therefore, it follows by part (c) of Lemma 3.1
with Lemma 1.1 that
\begin{equation}
\frac{1}{2}.\underset{n}{\lim \sup}\left \Vert \bar{A}_{n}-\bar{\alpha
}\right \Vert _{\ell_{q}}\leq \left \Vert L_{\bar{A}}\right \Vert _{\chi}%
\leq \underset{n}{\lim \sup}\left \Vert \bar{A}_{n}-\bar{\alpha}\right \Vert
_{\ell_{q}}, \tag{3.8}%
\end{equation}
where $\bar{\alpha}=(\bar{\alpha}_{k})$ with $\bar{\alpha}_{k}=\lim_{n}\bar
{a}_{nk}$ for all $k\in%
\mathbb{N}
$.

Now, let us write $S=S_{\ell_{p}(\widehat{F})}$ and $\bar{S}=S_{\ell_{p}},$
for short. Then, we obtain by (1.4) and Lemma 1.2 that%
\begin{equation}
\left \Vert L_{A}\right \Vert _{\chi}=\chi(L_{A}(S))=\chi(AS) \tag{3.9}%
\end{equation}
and%
\begin{equation}
\left \Vert L_{\bar{A}}\right \Vert _{\chi}=\chi(L_{\bar{A}}(\bar{S}))=\chi
(\bar{A}\bar{S}). \tag{3.10}%
\end{equation}

Further, we have by (2.3) that $x\in S$ if and only if $y\in \bar{S}$ and since
$Ax=\bar{A}y$ by Lemma 2.3, we deduce that $AS=\bar{A}\bar{S}.$ This leads us
with (3.9) and (3.10) to the consequence that $\left \Vert L_{A}\right \Vert
_{\chi}=\left \Vert L_{\bar{A}}\right \Vert _{\chi}.$ Hence, we get (3.5) from
(3.8). This completes the proof.
\end{proof}

The conclusions of Theorem 3.2 still hold for $\ell_{1}(\widehat{F})$ or
$\ell_{\infty}(\widehat{F})$ instead of $\ell_{p}(\widehat{F})$ with $q=1,$
and on replacing the summations over $k$ by the supremums over $k$ in the case
$\ell_{1}(\widehat{F})$. Then, we have the following results:

\begin{theorem}
Let $\bar{A}=(\bar{a}_{nk})$ be the associated matrix defined by (2.9). Then
we have

(a) If $A\in(\ell_{\infty}(\widehat{F}),\ell_{\infty})$, then
\[
0\leq \left \Vert L_{A}\right \Vert _{\chi}\leq \underset{n}{\lim \sup}%
{\displaystyle \sum \limits_{k}}
\left \vert \bar{a}_{nk}\right \vert
\]
and%
\[
L_{A}\text{ is compact if }\underset{n}{\lim}%
{\displaystyle \sum \limits_{k}}
\left \vert \bar{a}_{nk}\right \vert =0.
\]

(b) If $A\in(\ell_{\infty}(\widehat{F}),c_{0})$, then
\[
\left \Vert L_{A}\right \Vert _{\chi}=\underset{n}{\lim \sup}%
{\displaystyle \sum \limits_{k}}
\left \vert \bar{a}_{nk}\right \vert
\]
and%
\[
L_{A}\text{ is compact if and only if }\underset{n}{\lim}%
{\displaystyle \sum \limits_{k}}
\left \vert \bar{a}_{nk}\right \vert =0.
\]

(c) If $A\in(\ell_{\infty}(\widehat{F}),c)$, then%
\[
\frac{1}{2}.\underset{n}{\lim \sup}%
{\displaystyle \sum \limits_{k}}
\left \vert \bar{a}_{nk}-\bar{a}_{k}\right \vert \leq \left \Vert L_{A}\right \Vert
_{\chi}\leq \underset{n}{\lim \sup}%
{\displaystyle \sum \limits_{k}}
\left \vert \bar{a}_{nk}-\bar{a}_{k}\right \vert
\]
and%
\[
L_{A}\text{ is compact if and only if }\underset{n}{\lim}%
{\displaystyle \sum \limits_{k}}
\left \vert \bar{a}_{nk}-\bar{a}_{k}\right \vert =0,
\]
where $\bar{\alpha}=(\bar{\alpha}_{k})$ with $\bar{\alpha}_{k}=\lim_{n}\bar
{a}_{nk}$ for all $k\in%
\mathbb{N}
$.
\end{theorem}

\begin{theorem}
Let $\bar{A}=(\bar{a}_{nk})$ be the associated matrix defined by (2.9). Then
we have

(a) If $A\in(\ell_{1}(\widehat{F}),\ell_{\infty})$, then%
\[
0\leq \left \Vert L_{A}\right \Vert _{\chi}\leq \underset{n}{\lim \sup}\left(
\sup_{k}\left \vert \bar{a}_{nk}\right \vert \right)
\]
and%
\[
L_{A}\text{ is compact if }\underset{n}{\lim}\left(  \sup_{k}\left \vert
\bar{a}_{nk}\right \vert \right)  =0.
\]

(b) If $A\in(\ell_{1}(\widehat{F}),c_{0})$, then
\[
\left \Vert L_{A}\right \Vert _{\chi}=\underset{n}{\lim \sup}\left(  \sup
_{k}\left \vert \bar{a}_{nk}\right \vert \right)
\]
and%
\[
L_{A}\text{ is compact if and only if }\underset{n}{\lim}\left(  \sup
_{k}\left \vert \bar{a}_{nk}\right \vert \right)  =0.
\]

(c) If $A\in(\ell_{1}(\widehat{F}),c)$, then%
\[
\frac{1}{2}.\underset{n}{\lim \sup}\left(  \sup_{k}\left \vert \bar{a}_{nk}%
-\bar{a}_{k}\right \vert \right)  \leq \left \Vert L_{A}\right \Vert _{\chi}%
\leq \underset{n}{\lim \sup}\left(  \sup_{k}\left \vert \bar{a}_{nk}-\bar{a}%
_{k}\right \vert \right)
\]
and%
\[
L_{A}\text{ is compact if and only if }\underset{n}{\lim}\left(  \sup
_{k}\left \vert \bar{a}_{nk}-\bar{a}_{k}\right \vert \right)  =0,
\]
where $\bar{\alpha}=(\bar{\alpha}_{k})$ with $\bar{\alpha}_{k}=\lim_{n}\bar
{a}_{nk}$ for all $k\in%
\mathbb{N}
$.
\end{theorem}

Morever, as an immediate consequence of Theorem 3.3, we have the following corollary.

\begin{corollary}
If either $A\in(\ell_{\infty}(\widehat{F}),c_{0})$ or $A\in(\ell_{\infty
}(\widehat{F}),c)$, then the operator $L_{A}$ is compact.
\end{corollary}

\begin{proof}
Let $A\in(\ell_{\infty}(\widehat{F}),c_{0}).$ Then, we have by Lemma 2.3 that
$\bar{A}\in(\ell_{\infty},c_{0})$ which implies that $\lim_{n}\left(  \sum
_{k}\left \vert \bar{a}_{nk}\right \vert \right)  =0,$ [39]. This leads us with
Theorem 3.3(b) to the consequence that $L_{A}$ is compact. Similarly, if
$A\in(\ell_{\infty}(\widehat{F}),c)$ then $\bar{A}\in(\ell_{\infty},c)$ and
hence $\lim_{n}\left(  \sum_{k}\left \vert \bar{a}_{nk}-\bar{a}_{k}\right \vert
\right)  =0$, where $\bar{\alpha}=(\bar{\alpha}_{k})$ with $\bar{\alpha}%
_{k}=\lim_{n}\bar{a}_{nk}$ for all $k\in%
\mathbb{N}
$. Therefore, we deduce from Theorem 3.3(c) that $L_{A}$ is compact.
\end{proof}

Throughout, let $%
\mathcal{F}%
_{m}$ $(m\in%
\mathbb{N}
)$ be the subcollection of $%
\mathcal{F}%
$ consisting of all nonempty and finite subsets of $%
\mathbb{N}
$ with elements that are greater than $m$, that is%
\[%
\mathcal{F}%
_{m}=\left \{  N\in%
\mathcal{F}%
:n>m\text{ for all }n\in%
\mathbb{N}
\right \}  ;\text{ }(m\in%
\mathbb{N}
).
\]

The next lemma [25, Theorem 3.11] gives necessary and sufficient conditions
for a matrix transformation from a $BK$ space to $\ell_{1}$ to be compact.

\begin{lemma}
Let $X\supset \varphi$ be a $BK$ space. If $A\in(X,\ell_{1})$, then
\[
\lim_{m}\left(  \sup_{N\in%
\mathcal{F}%
_{m}}\left \Vert
{\displaystyle \sum \limits_{n\in N}}
A_{n}\right \Vert _{X}^{\ast}\right)  \leq \left \Vert L_{A}\right \Vert _{\chi
}\leq4.\lim_{m}\left(  \sup_{N\in%
\mathcal{F}%
_{m}}\left \Vert
{\displaystyle \sum \limits_{n\in N}}
A_{n}\right \Vert _{X}^{\ast}\right)  .
\]

\end{lemma}

Now, we prove the following result:

\begin{theorem}
Let $1<p<\infty$ and $q=p/(p-1)$. If $A\in(\ell_{p}(\widehat{F}),\ell_{1})$,
then%
\begin{equation}
\lim_{m}\left \Vert A\right \Vert _{(\ell_{p}(\widehat{F}),\ell_{1})}^{(m)}%
\leq \left \Vert L_{A}\right \Vert _{\chi}\leq4.\lim_{m}\left \Vert A\right \Vert
_{(\ell_{p}(\widehat{F}),\ell_{1})}^{(m)} \tag{3.11}%
\end{equation}
and%
\begin{equation}
L_{A}\text{ is compact if and only if }\lim_{m}\left \Vert A\right \Vert
_{(\ell_{p}(\widehat{F}),\ell_{1})}^{(m)}=0, \tag{3.12}%
\end{equation}
where
\[
\left \Vert A\right \Vert _{(\ell_{p}(\widehat{F}),\ell_{1})}^{(m)}=\sup_{N\in%
\mathcal{F}%
_{m}}\left(
{\displaystyle \sum \limits_{k}}
\left \vert
{\displaystyle \sum \limits_{n\in N}}
\bar{a}_{nk}\right \vert ^{q}\right)  ^{1/q};\text{ \ }(m\in%
\mathbb{N}
).
\]

\end{theorem}

\begin{proof}
It is obvious that (3.11) is obtained by combining Lemmas 2.2(c), 2.3 and 3.6.
Also, by using (1.5), we get (3.12) from (3.11).
\end{proof}

\begin{theorem}
Let $1\leq p<\infty$. If $A\in(\ell_{1}(\widehat{F}),\ell_{p})$, then%
\begin{equation}
\left \Vert L_{A}\right \Vert _{\chi}=\lim_{m}\left(  \sup_{k}\left(
{\displaystyle \sum \limits_{n=m}^{\infty}}
\left \vert \bar{a}_{nk}\right \vert ^{p}\right)  ^{1/p}\right)  \tag{3.13}%
\end{equation}
and%
\begin{equation}
L_{A}\text{ is compact if and only if }\lim_{m}\left(  \sup_{k}\left(
{\displaystyle \sum \limits_{n=m}^{\infty}}
\left \vert \bar{a}_{nk}\right \vert ^{p}\right)  ^{1/p}\right)  =0, \tag{3.14}%
\end{equation}

\end{theorem}

\begin{proof}
Let us remark that the limit in (3.13) exists by Lemma 2.4.

Now, we write $S=S_{\ell_{1}(\widehat{F})}$. Then, we have by Lemma 1.2 that
$L_{A}(S)=AS\in M_{\ell_{p}}$. Thus, it follows from (1.4) and Lemma 1.4 that%
\begin{equation}
\left \Vert L_{A}\right \Vert _{\chi}=\chi(AS)=\lim_{m}\left(  \sup_{x\in
S}\left \Vert (I-P_{m})(Ax)\right \Vert _{\ell_{p}}\right)  , \tag{3.15}%
\end{equation}
where $P_{m}:\ell_{p}\rightarrow \ell_{p}$ $(m\in%
\mathbb{N}
)$ is the operator defined by $P_{m}(x)=(x_{0},x_{1},...,x_{m},0,0,...)$ for
all $x=(x_{k})\in \ell_{p}$ and $I$ is the identity operator on $\ell_{p}$.

On the other hand, let $x\in \ell_{1}(\widehat{F})$ be given. Then $y\in
\ell_{1}$ and since $A\in(\ell_{1}(\widehat{F}),\ell_{p})$, we obtain from
Lemma 2.3 that $\bar{A}\in(\ell_{1},\ell_{p})$ and $Ax=\bar{A}y.$ Thus, we
have for every $m\in%
\mathbb{N}
$ that%
\begin{align*}
\left \Vert (I-P_{m})(Ax)\right \Vert _{\ell_{p}}  &  =\left \Vert (I-P_{m}%
)(\bar{A}y)\right \Vert _{\ell_{p}}\\
&  =\left(
{\displaystyle \sum \limits_{n=m+1}^{\infty}}
\left \vert \bar{A}_{n}(y)\right \vert ^{p}\right)  ^{1/p}\\
&  =\left(
{\displaystyle \sum \limits_{n=m+1}^{\infty}}
\left \vert
{\displaystyle \sum \limits_{k}}
\bar{a}_{nk}y_{k}\right \vert ^{p}\right)  ^{1/p}\\
&  \leq%
{\displaystyle \sum \limits_{k}}
\left(
{\displaystyle \sum \limits_{n=m+1}^{\infty}}
\left \vert \bar{a}_{nk}y_{k}\right \vert ^{p}\right)  ^{1/p}\\
&  \leq \left \Vert y\right \Vert _{\ell_{1}}\left(  \sup_{k}\left(
{\displaystyle \sum \limits_{n=m+1}^{\infty}}
\left \vert \bar{a}_{nk}\right \vert ^{p}\right)  ^{1/p}\right) \\
&  =\left \Vert x\right \Vert _{\ell_{1}(\widehat{F})}\left(  \sup_{k}\left(
{\displaystyle \sum \limits_{n=m+1}^{\infty}}
\left \vert \bar{a}_{nk}\right \vert ^{p}\right)  ^{1/p}\right)  .
\end{align*}
This yields that
\[
\sup_{x\in S}\left \Vert (I-P_{m})(Ax)\right \Vert _{\ell_{p}}\leq \sup
_{k}\left(
{\displaystyle \sum \limits_{n=m+1}^{\infty}}
\left \vert \bar{a}_{nk}\right \vert ^{p}\right)  ^{1/p}\text{; \ }(m\in%
\mathbb{N}
).
\]
Therefore, we deduce from (3.15) that
\begin{equation}
\left \Vert L_{A}\right \Vert _{\chi}\leq \lim_{m}\left(  \sup_{k}\left(
{\displaystyle \sum \limits_{n=m+1}^{\infty}}
\left \vert \bar{a}_{nk}\right \vert ^{p}\right)  ^{1/p}\right)  . \tag{3.16}%
\end{equation}

To prove the converse inequality, let $c^{(k)}\in \ell_{1}(\widehat{F})$ be
such that $\widehat{F}c^{(k)}=e^{(k)}$ $(k\in%
\mathbb{N}
)$, that is, $e^{(k)}$ is the $\widehat{F}$-transform of $c^{(k)}$ for each
$k\in%
\mathbb{N}
$. Then, we have by Lemma 2.3 that $Ac^{(k)}=\bar{A}e^{(k)}$ for every $k\in%
\mathbb{N}
$.

Now, let $U=\{c^{(k)}:$ $k\in%
\mathbb{N}
\}$. Then $U\subset S$ and hence $AU\subset AS$ which implies that
$\chi(AU)\leq \chi(AS)=\left \Vert L_{A}\right \Vert _{\chi}$.

Further, it follows by applying Lemma 1.4 that
\begin{align*}
\chi(AU)  &  =\lim_{m}\left(  \sup_{k}\left(  \left(
{\displaystyle \sum \limits_{n=m+1}^{\infty}}
\left \vert A_{n}(c^{(k)})\right \vert ^{p}\right)  ^{1/p}\right)  \right) \\
&  =\lim_{m}\left(  \sup_{k}\left(
{\displaystyle \sum \limits_{n=m+1}^{\infty}}
\left \vert \bar{a}_{nk}\right \vert ^{p}\right)  ^{1/p}\right)  .
\end{align*}

Thus, we obtain that
\begin{equation}
\lim_{m}\left(  \sup_{k}\left(
{\displaystyle \sum \limits_{n=m+1}^{\infty}}
\left \vert \bar{a}_{nk}\right \vert ^{p}\right)  ^{1/p}\right)  \leq \left \Vert
L_{A}\right \Vert _{\chi}. \tag{3.17}%
\end{equation}

Hence, we get (3.13) by combining (3.16) and (3.17). This completes the proof,
since (3.14) is immediate by (1.5) and (3.13).
\end{proof}

\end{document}